\documentclass[oneside,11pt]{amsart}
\usepackage{amsmath,amsfonts, latexsym,amssymb, times}
\usepackage[active]{srcltx}
\newtheorem{theorem}{Theorem}[section]
\newtheorem{lemma}[theorem]{Lemma}

\theoremstyle{definition}
\newtheorem{definition}[theorem]{Definition}

\newcommand{\eif}{e^{i\varphi}}

\numberwithin{equation}{section}

 \makeatletter
\renewcommand*\subjclass[2][2000]{%
  \def\@subjclass{#2}%
  \@ifundefined{subjclassname@#1}{%
    \ClassWarning{\@classname}{Unknown edition (#1) of Mathematics
      Subject Classification; using '1991'.}%
  }{%
    \@xp\let\@xp\subjclassname\csname subjclassname@#1\endcsname
  }%
}
 \makeatother
\newcounter{minutes}
\divide\time by 60
\newcounter{hours}
\multiply\time by 60 \addtocounter{minutes}{-\time}

\begin{document}

\title[Quasiconformal and harmonic mappings]{Quasiconformal harmonic mappings between Dini's smooth Jordan domains}
\author{David Kalaj}
\address{ Faculty of natural sciences and mathematics, University of Montenegro,
Dzordza Vasingtona b.b. 81000, Podgorica, Montenegro}
\email{davidk@ac.me} \subjclass {Primary 30C55, Secondary 31C05}
\keywords{Planar harmonic mappings, Quasiconformal, Dini smooth}

\begin{abstract}
Let $D$ and $\Omega$ be Jordan domains with Dini's smooth boundaries and
and let $f:D\mapsto \Omega$ be a harmonic homeomorphism.
The object of the paper is to prove the following result: If
$f$ is quasiconformal, then $f$ is Lipschitz. This extends some recent results, where stronger assumptions on the boundary are imposed, and somehow is optimal, since it coincides with the best condition for Lipschitz behavior of  conformal mappings in the plane and conformal parametrization of  minimal surfaces.
\end{abstract}
\maketitle
\def\thefootnote{}
\footnotetext{ \texttt{\tiny File:~\jobname .tex,
          printed: \number\year-\number\month-\number\day,
          \thehours.\ifnum\theminutes<10{0}\fi\theminutes }
} \makeatletter\def\thefootnote{\@arabic\c@footnote}\makeatother
\tableofcontents
\section{Introduction and statement of the main result}
\subsection{Quasiconformal mappings}  By definition,  $K$-quasiconformal mappings (or shortly q.c. mappings) are orientation preserving homeomorphisms $f:D \to \Omega$ between domains $D,\Omega \subset \mathbf{C}$, contained  in the Sobolev class $W^{1,2}_{loc}(D)$, for which  the differential matrix   and its determinant are coupled in the distortion inequality,
\begin{equation}\label{distortion}
 |D\!f(z)|^2  \leq K\, \det D\!f(z)\;,\quad \textrm{where}\;\;\;|D\!f(z)|  = \max_{|\xi| =1} \; |D\!f(z) \xi|,
\end{equation}
for some  $K \geq1$. Here $\det D\!f(z)$ is the determinant of the formal derivative $D\!f(z)$, which will be denoted in the sequel by $J_f(z)$.  Note that the
condition (\ref{distortion}) can be written in complex notation as \begin{equation}\label{same}{(|f_z|+|f_{\bar z}|)^2}\le K({|f_z|^2-|f_{\bar z}^2|})\quad \text{a.e. on $D$}\end{equation} or what is the same  $$|f_{\bar z}|\le
k|f_z|\quad \text{a.e. on $D$ where $k=\frac{K-1}{K+1}$ i.e.
$K=\frac{1+k}{1-k}$ }.$$

\subsection{Harmonic mappings}
A mapping  $f$ is called \emph{harmonic}
in a region $D$ if it has the form  $f=u+iv$ where $u$ and $v$ are
real-valued harmonic functions in $D$. If $D$ is simply-connected,
then there are two analytic functions $g$ and $h$ defined on $D$
such that $f$ has the representation $$f=g+\overline h.$$

If $f$ is a harmonic univalent function, then by Lewy's theorem
(see \cite{l}), $f$ has a non-vanishing Jacobian and consequently,
according to the inverse mapping theorem, $f$ is a diffeomorphism.

Let $$P(r,x-\varphi)=\frac{1-r^2}{2\pi
(1-2r\cos(x-\varphi)+r^2)}$$ denote the Poisson kernel. If $F\in L^1(\mathbf{T})$, where $\mathbf{T}$ is the unit circle, then we define the Poisson integral $\mathcal{P}[F]$ of $F$ by formula \begin{equation}\label{e:POISSON}
\mathcal{P}[F](z)=\int_0^{2\pi}P(r,x-\varphi)F(e^{ix})dx, \ \ |z|<1,\ \ z=re^{i\varphi}.
\end{equation} The function $f(z)=\mathcal{P}[F](z)$ is a harmonic mapping in the unit disk $\mathbf{U}=\{z:|z|<1\}$, which belongs to the Hardy space $h^1(\mathbf{U})$. The mapping $f$ is bounded in $\mathbf{U}=\{z:|z|<1\}$ if and only if $F\in L^\infty(\mathbf{T}).$
 Standard properties of the Poisson integral show that $\mathcal{P}[F]$ extends by continuity to $F$ on $\overline{\mathbf{U}}$, provided that $F$ is continuous. For this facts and standard properties of harmonic Hardy space we refer to \cite[Chapter~6]{abr} and \cite{Pd}. With the additional assumption that $F$ is orientation-preserving homeomorphism of this circle onto a convex Jordan curve $\gamma$, $\mathcal{P}[F]$ is an orientation preserving diffeomorphism of the open unit disk. This is indeed the celebrated theorem of Choquet-Rado-Kneser (\cite{knes,duren}).
This theorem is not true for non-convex domains, but hold true under some additional assumptions.
It has been extended in various
directions (see for example \cite{jj},  \cite{studia} and
\cite{dur}).

\subsection{Hilbert transform}
The Hilbert transform of a function $\chi\in L^1(\mathbf T)$
is defined by the formula
\begin{equation}\label{hilbert}\tilde\chi(\tau)=H(\chi)(\tau)=-\frac 1\pi
\int_{0^+}^\pi\frac{\chi(\tau+t)-\chi(\tau-t)}{2\tan(t/2)}\mathrm
dt.\end{equation} Here $\int_{0^+}^\pi \Phi(t) dt:=\lim_{\epsilon\to 0^+}
\int_{\epsilon}^\pi \Phi(t) dt.$ This integral is improper and
converges for a.e. $\tau\in[0,2\pi]$; this and other facts
concerning the operator $H$ used in this paper can be found in the
book of Zygmund \cite[Chapter~VII]{ZY}. If $f=u+iv$ is a harmonic
function defined in the unit disk $\mathbf{U}$ then a harmonic function $\tilde f=\tilde u+i\tilde v$ is called the harmonic
conjugate of $f$ if $u+i\tilde u$ and $v+i\tilde v$ are analytic functions and $\tilde f(0)=0$.
 Let
$\chi, \tilde\chi\in L^1(\mathbf T)$. Then
\begin{equation}\label{lit}\mathcal{P}[\tilde \chi]=\widetilde
{\mathcal{P}[\chi]},\end{equation} where $\tilde k(z)$ is the harmonic
conjugate of $k(z)$ (see e.g. \cite[Theorem~6.1.3]{lib}).

 If $f=u+iv$ is a harmonic
function defined in a Dini smooth Jordan domain $D$ then a harmonic function $\tilde f=\tilde u+i\tilde v$ is called the harmonic
conjugate of $f$ if $u+i\tilde u$ and $v+i\tilde v$ are analytic functions. Notice that $\tilde f$ is uniquely determined up to an additive constant.
 Let $\Phi: D\to \mathbf{U}$ be a conformal mapping, and $G\in L^1(\partial D)$. Then the Poisson integral w.r.t. domain $D$ of $G$ is defined by
$$\mathcal{P}_D[G](z)=\frac{1}{2\pi}\int_{\partial D}\frac{1-|\Phi(z)|^2}{|\Phi(z)-\Phi(\zeta)|^2}G(\zeta) |\Phi'(\zeta)|d\zeta|.$$

Let
$\chi$ be the boundary value of $f$ and assume that $\tilde \chi$ is the boundary value of $\tilde f$. Then $\tilde \chi$ is called the Hilbert transform of $\chi$, i.e. $\tilde\chi=H(\chi)$. Assume that $\tilde\chi\in L^1(\partial D)$. Since $\mathcal{P}_D[G](z)=\mathcal{P}[G\circ \Phi^{-1}](\Phi(z))$, we have
\begin{equation}\label{lita}\mathcal{P}_D[\tilde \chi]=\widetilde
{\mathcal{P}_D[\chi]},\end{equation} where $\tilde k(z)$ is the harmonic
conjugate of $k(z)$ (c.f. \eqref{lit}).

If $f=g+\overline{h}:\mathbf{U}\to \Omega$ is a harmonic function  mapping then the radial and tangential derivatives at $z=re^{it}$ are defined by $$\partial_r f(z)=\frac{1}{r}(g'+\overline{h'})$$ and $$\partial_t f(z)=i(g'-\overline{h'}).$$ So $r \partial_r f$ is the harmonic conjugate of $\partial_t f$. We generalize this definition for a mapping $f=g+\overline{h}$ defined in a Jordan domain $D$. In order to do so, let $\Phi=Re^{i\Theta}$ be a conformal mapping of the domain $D$ onto the unit disk. Then the radial derivative  and tangent derivative of $f$ in a point $w\in D$ are defined by $$\partial_R f(w)=\frac{1}{|\Phi(w)|}Df(w)\left(\frac{\Phi(w)}{\Phi'(w)}\right)$$

$$\partial_\Theta f(w)=Df(w)\left(i\frac{\Phi(w)}{\Phi'(w)}\right).$$ Here $\frac{\Phi(w)}{\Phi'(w)}$ and $i\frac{\Phi(w)}{\Phi'(w)}$ are treated as two vectors from $\mathbf{R}^2\cong \mathbf{C}$. Then it is easy to show that $$R\partial_R f(w)=\frac{g'(w)}{\Phi'(w)}+{\frac{\overline{h'(w)}}{\overline{\Phi'(w)}}}$$ and $$\partial_\Theta f(w)=i\left(\frac{g'(w)}{\Phi'(w)}-{\frac{\overline{h'(w)}}{\overline{\Phi'(w)}}}\right).$$ This implies that $R\partial_R f(w)$ and $\partial_\Theta f(w)$ are harmonic functions in $D$ and $R\partial_R f(w)$ is the harmonic conjugate of $\partial_\Theta f(w)$. Notice also that, these derivatives are uniquely determined up to a conformal mapping $\Phi$. Assume further that $D$ and $\Omega$ have Dini smooth boundaries. If $F:\partial D\to\partial\Omega$ is the boundary function of $f$, and $\partial_\Theta f(w)$ is a bounded harmonic function, then  $$\lim_{w\to w_0} \partial_\Theta f(w)=F'(w_0),$$ where the limit is non-tangential. Here $$F'(w_0):=\frac{\partial (F\circ \Phi^{-1})(e^{it})}{\partial t}$$ where $\Phi(w_0)=e^{it}$. If $F'\in L^1(\partial D)$, then the harmonic function $R\partial_R f(w)$ has non-tangential limits in almost every point of $\partial D$ and its boundary value is the Hilbert transform of $F'$, namely $$H(F')(w_0)=\lim_{w\to w_0} R\partial_R f(w).$$

From now on the boundary value
of $f$ will be denoted by $F$. We will focus on orientation-preserving harmonic
quasiconformal mappings between smooth domains and investigate their Lipschitz character up to the boundary. For future reference, we will say that a q.c. mapping $f:\mathbf U\to \Omega$ of the unit disk onto the Jordan domain $\Omega$ with rectifiable boundary  is
\emph{normalized} if $f(1)=w_0$, $f(e^{2\pi i/3}) = w_1$ and $f(e^{4\pi i/3
}) = w_2$, where ${w_0w_1}$, $w_1w_2$ and $w_2w_0$ are arcs of
$\gamma=\partial \Omega$ having the same length $|\gamma|/3$.
\subsection{Background}  Let $\Omega$ be a Jordan domain with rectifiable boundary, and $\gamma(t)$
the arc-length parametrization of $\partial \Omega$. We say that $\partial\Omega$ is $C^1$ if $\gamma\in C^1$. Then $\arg\gamma'$ is continuous and let $\omega$ be its modulus of continuity.
If $\omega$ satisfies \begin{equation}\label{dini}\int_0^\delta \frac{\omega(t)}{t}dt<\infty \ \  (\delta>0);\end{equation}
we say that $\partial\Omega$ is Dini smooth. Denote by $C^{1,\varpi}$ the class of all Dini smooth Jordan curves. The derivative of a conformal mapping $f$ of the unit disk onto $\Omega$ is continuous and non-vanishing in $\overline{D}$ \cite[Theorem~10.2]{PW} (see also \cite{proc}).
 This implies that $f$ is bi-Lipschitz continuous. For the later reference we refer to this result as the Kellogg theorem, who was the first to prove this result for $C^{1,\alpha}$, domains, where $0<\alpha<1$. Warschawski in \cite{arma} proved the same result for conformal parametrization of a minimal surface.

If $f$ is merely quasiconformal and maps the unit disk onto itself, then Mori theorem implies that $|f(z)-f(w)|\le M_1(K)|z-w|^{1/K}$. The constant $1/K$ is the best possible. If $f$ is a conformal mapping of the unit disk onto a Jordan domain with merely $C^1$ boundary, then the function $f$ is not necessarily Lipschitz (see for example the paper of Lesley and Warschawski \cite[p.~277]{Lw}). This is why we need to add some assumption, other than quasiconformality, as well as some smoothness of image curve which is better than $C^1$, in order to obtain that the resulting mapping is Lipschitz or bi-Lipschitz.

Since every conformal mapping in the plane is harmonic and quasiconformal, it is an interesting question to what extend the smoothness of the boundary of a Jordan domain $\Omega$ implies that the quasiconformal harmonic mapping of the unit disk onto $\Omega$ is Lipschitz.  The first study of harmonic
quasiconformal mappings of the unit disk onto itself has been done by O. Martio \cite{Om}. By using Heinz inequality \cite{HE}, Martio gave some sufficient conditions on a diffeomorphic self-mapping $F$ of the unit circle such that its harmonic extension $\mathcal{P}$ is quasiconformal. This paper has been generalized by the author in \cite{kalajpub} for q.c. mappings from the unit disk onto a convex Jordan domain. Pavlovi\'c in \cite{MP} proved that every q.c. harmonic mapping of the unit disk onto itself is Lipschitz, providing very clever  proof. Kalaj in \cite{kalajd} proved that every q.c. harmonic mapping between two Jordan domains with $C^{1,\alpha}$ boundary is Lipschitz. This result has its counterpart for non-euclidean metrics \cite{km}. For a generalization of the last result to several-dimensional case we refer to the paper \cite{jda}. The problem of bi-Lipschitz continuity of a quasiconformal mapping of the unit disk onto a Jordan domain with $C^2$ boundary has been solved in \cite{pisa}. The object of
the present paper is to extend some of these results.
\subsection{New results}
The
following theorem is such an extension in which the H\"older continuity is replaced
by the more general Dini condition.
\begin{theorem}\label{main}
Let $f=\mathcal{P}[F](z)$ be a harmonic normalized $K$ quasiconformal mapping between the unit
disk and the Jordan domain $\Omega$ with
$\gamma=\partial \Omega\in C^{1,\varpi}$. Then there exists a
constant $C'=C'(\gamma,K)$ such that
\begin{equation}\label{desired}\left|\frac{\partial F(e^{i\varphi})}{\partial \varphi}\right|\le C'\text{ for almost every
$\varphi\in [0,2\pi]$},\end{equation}

and

\begin{equation}\label{equati}|f(z_1)-f(z_2)|\le KC'|z_1-z_2|\,\,
\text{ for }z_1,z_2\in \mathbf U.\end{equation}
\end{theorem}
By using Theorem~\ref{main}, we obtain the following improvement of \cite[Theorem~3.1]{kalajd}.
 \begin{theorem}\label{theon} {\it Let $D$ and $\Omega$ be Jordan
domains such that $\partial D, \partial \Omega\in
C^{1,\varpi}$ and let $f:D\mapsto \Omega$ be
a harmonic homeomorphism. The following statements hold true.
\begin{enumerate}
\item[(a)] If   $f$ is quasiconformal, then $f$ is Lipschitz.
\item[(b)] If   $\Omega$ is convex and $f$ is q.c., then $f$
is bi-Lipschitz.
\item[(c)] If
$\Omega$ is convex,  then $f$ is q.c. if and only $\log |F'|,  H(F')\in L^\infty(\partial D)$.
\end{enumerate}

}
\end{theorem}
\begin{proof}[Proof of Theorem~\ref{theon}]
(a) Chose a conformal mapping $\Phi: \mathbf{U}\to D$ and define $f_1=f\circ \Phi$. Then $f_1$ is a q.c. harmonic mapping of the unit disk onto $\Omega'$, so it satisfies the conditions of Theorem~\ref{main}. This implies in particular that $f_1$ is Lipschitz. In view of Kellogg theorem, the mapping $\Phi$ is bi-Lipschitz. Thus $f=f_1\circ h^{-1}$ is Lipschitz.

(b) If $\Omega$ is a convex domain, and $D=\mathbf{U}$ then by a result of the author (\cite{Dk}) we have that $$|D f(z)|\ge \frac{1}{4}\mathrm{dist}(f(0), \partial\Omega')$$ for $z\in \mathbf{U}$. If $\Omega$ is not the unit disk, then we make use of the conformal mapping $\Phi:\mathbf{U}\to \Omega$ as in the proof of (a). Then we obtain $$ |Df(z)|=|Df_1(z)|/|\Phi'(z)|\ge c.$$ Now by using the quasiconformality of $f$, we have that $$|Df(z)|^2\le K J_f(z).$$ Therefore $$J_{f^{-1}(f(z))}=\frac{1}{J_f(z)}\le \frac{K}{c^2}.$$ Since $f^{-1}$ is $K-$quasiconformal, we have further that
$$|Df^{-1}(w)|^2\le K J_{f^{-1}}(w)\le \frac{K^2}{c^2}.$$ This implies that $f^{-1}$ is Lipschitz. This finishes the proof of (b).

(c) If $f$ is harmonic and quasiconformal, then by (b) it is bi-Lipschitz, and so its boundary function $F$ is bi-Lipschitz. Further $R\partial_R f$ is bounded harmonic function and this is equivalent with the fact that $\log|F'|\in L^\infty(\partial D)$. Since $H(F')$ is its boundary function, it is bounded, i.e. it belongs to $L^\infty(\partial D)$.

Prove now the opposite implication. Since $$\partial_\Theta f=\mathcal{P}_D[F']\text{ and }\, R\partial_R f =\mathcal{P}_D[H(F')],$$ it follows that $\partial_\Theta f$ and $R\partial_R f$ are bounded harmonic functions. This means that $|Df|$ is bounded by a constant $M$. In order to show that $f$ is quasiconformal, it is enough to show that the Jacobian of $f$ is bigger than a positive constant in $D$. Let $f_1=f\circ \Phi^{-1}$, and let $\delta=\mathrm{dist}(f_1(0),\partial \Omega)$ and $\kappa=\min|\partial_tf_1(e^{it})|$. Then by \cite[Corollary~2.9]{kalajpub}, we have $$J_f(\Phi(w))|\Phi'(w)|^2=J_{f_1}(w)\ge \frac{\kappa\delta}{2}.$$ So $$J_f(z)\ge c>0, \ \ z\in D.$$ We conclude that $$\frac{|Df(z)|^2}{J_f(z)}\le \frac{M^2}{c}.$$ This finishes the proof.

\end{proof}

\section{Preliminary results}

\begin{definition}\label{defi}
Let $\xi:[a,b]\to \mathbf C$ be a continuous function. The modulus of
continuity of $\xi$ is
   $$\omega(t)= \omega_\xi(t) = \sup_{|x-y|\le t} |\xi(x)-\xi(y)|. \,$$
The function $\xi$ is called Dini continuous if
\begin{equation}\label{didi}\int_{0}^{b-a} \frac{\omega_\xi(t)}{t}\,dt < \infty.\end{equation}   A smooth
Jordan curve $\gamma$ with the length $l=|\gamma|$, is said to be
Dini smooth if $g'$ is Dini continuous on $[0,l]$. If $\omega(t)$, $0\le t\le l$ is the modulus of continuity of $g'$, then we extend  $\omega(t)=\omega(l)$ for $t\ge l$.

A function $F:\mathbf{T}\to \gamma$ is called Dini smooth if the function $\Phi(t)=F(e^{it})$ is Dini smooth, i.e. $$|\Phi'(t)-\Phi'(s)|\le \omega(|t-s|),$$ where $\omega$ is Dini continuous. Observe that every smooth
$C^{1,\alpha}$ Jordan curve is Dini smooth.
\end{definition}
Let
\begin{equation}\label{ker}\mathcal{K}(s,t)=\text{Re}\,[\overline{(g(t)-g(s))}\cdot i
g'(s)]\end{equation} be a function defined on $[0,l]\times[0,l]$. By
$\mathcal{K}(s\pm l, t\pm l)=\mathcal{K}(s,t)$ we extend it on $\mathbf R\times \mathbf
R$.  Suppose now that $\Psi:\mathbf R\mapsto \gamma$
is an arbitrary $2\pi$ periodic Lipschitz function such that
$\Psi|_{[0,2\pi)}:[0,2\pi)\mapsto \gamma$ is an orientation preserving
bijective function. { Then there exists an increasing continuous
function $\psi:[0,2\pi]\mapsto [0,l]$ such that
\begin{equation}\label{fgs}\Psi(\tau)=g(\psi(\tau)).\end{equation}}
We have for a.e. $e^{i\tau}\in \mathbf T$
that
$$\Psi'(\tau)=g'(\psi(\tau))\cdot \psi'(\tau),$$ and therefore
$$|\Psi'(\tau)|=|g'(\psi(\tau))|\cdot |\psi'(\tau)|=\psi'(\tau).$$

 Along with the function $\mathcal{K}$ we will also consider the function $\mathcal{K}_F$
defined by
$$\mathcal{K}_F(t,\tau)=\text{Re}\,[\overline{(\Psi(t)-\Psi(\tau))}\cdot
i\Psi'(\tau)].$$ Here  $F(e^{it})=\Psi(t)$. It is easy to see that
\begin{equation}\label{kg}\mathcal{K}_F(t,\tau)=\psi'(\tau)\mathcal{K}(\psi(t), \psi(\tau)).
\end{equation}

\begin{lemma}\label{le11}
Let $\gamma$ be a Dini smooth Jordan curve and $g:[0,l]\mapsto
\gamma$ be a natural parametrization of a Jordan curve with
$g'$ having modulus of continuity $\omega$.   Assume further that  $\Psi:[0,2\pi]\mapsto
\gamma$ is an arbitrary parametrization of $\gamma$ and let $F(e^{it})=\Psi(t)$. Then
\begin{equation}\label{kernel}|\mathcal{K}(s, t)|\le \int_0^{\min\{|s-t|,l-|s-t|\}}\omega(\tau)d\tau\end{equation}
and
\begin{equation}\label{kernelar}|\mathcal{K}_F(\varphi, x)|\le
|\psi'(\varphi)|
\int_0^{d_\gamma(\Psi(\varphi),\Psi(x))}\omega(\tau)d\tau.\end{equation}
 Here
$d_\gamma(\Psi(\varphi),\Psi(x)):= \min\{|s(\varphi)-s(x)|,
(l-|s(\varphi)-s(x)|)\}$ is the distance (shorter) between
$\Psi(\varphi)$ and $\Psi(x)$ along $\gamma$ which satisfies the
relation $$|\Psi(\varphi)-\Psi(x)|\le
d_\gamma(\Psi(\varphi),\Psi(x))\le
B_\gamma|\Psi(\varphi)-\Psi(x)|.$$

\end{lemma}
\begin{proof}
Note that the estimate \eqref{kernel} has been proved in \cite[Lemma~2.3]{studia}. Now \eqref{kernelar} follows from \eqref{kernel} and \eqref{kg}.

\end{proof}
A closed rectifiable Jordan curve $\gamma$ enjoys a $B-$ chord-arc
condition for some constant $B> 1$ if for all $z_1,z_2\in \gamma$
\begin{equation}\label{24march}
d_\gamma(z_1,z_2)\le B|z_1-z_2|.
\end{equation}
It is clear that if $\gamma\in C^{1}$, then $\gamma$ enjoys a
chord-arc condition for some $B_\gamma>1$.
The following lemma is proved in \cite{mana}.
\begin{lemma}\label{newle}
Assume that $\gamma$ enjoys a chord-arc condition for some $B>1$.
Then for every normalized  $K-$ q.c. mapping $f$ between the unit
disk $\mathbf U$ and the Jordan domain $\Omega=\mathrm{int}\gamma$ we have
$$|f(z_1)-f(z_2)|\le \Lambda_\gamma(K)|z_1-z_2|^\alpha,\quad z_1,z_2\in
\mathbf T,$$ where $$\alpha = \frac{2}{K(1+2B)^2},\quad \Lambda_\gamma(K)
=4\cdot2^{\alpha} (1+2B)\sqrt{\frac{2\pi K|\Omega|}{\log 2}}.$$
\end{lemma}
In the following lemma, there were  given some estimates for the Jacobian of
a harmonic univalent function.

\begin{lemma}\cite[Lemma~3.1]{studia}\label{le13}
If $f=\mathcal{P}[F]$ is a harmonic mapping, such that $F$ is a Lipschitz
homeomorphism from the unit circle onto a Dini smooth Jordan curve. Let $g$ be arc-length parametrization and assume that $\Psi(t)=F(e^{it})=g(\psi(t))$.
Then for almost every $\tau\in [0,2\pi]$ there exists
$$J_f(e^{i\tau}) :=\lim_{r\to 1} J_f(re^{i\tau})$$ and there hold the
formula

\begin{equation}\label{jacfor}\begin{split}J_f(e^{i\tau})&= \psi'(\tau)\int_0^{2\pi}
\frac{\mathrm{Re}\,[\overline{(g(\psi(t))-g(\psi(\tau)))}\cdot i
g'(\psi(\tau)))]}{2\sin^2\frac{t-\tau}{2}}\frac{dt}{2\pi}.\end{split}\end{equation}
\end{lemma}
From Lemma~\ref{le11} and Lemma~\ref{le13} we obtain
\begin{lemma}\label{jacabove}
Under the conditions and notation of Lemma~\ref{le13} we have
\begin{equation}\label{jakobo} J_f(e^{i\varphi}) \leq
\frac{\pi}{4}|\Psi'(\varphi)|\int_{-\pi}^\pi \frac 1{x^2}
{\int_0^{d_\gamma(F(e^{i(\varphi+x)}),F(e^{i\varphi}))}\omega(\tau)d\tau}
dx\end{equation} for a.e. $\eif \in \mathbf{T}$. Here $\omega$ is the  modulus of continuity of $g'$.
\end{lemma}

\begin{lemma}\label{jacabove2}
Let $f=\mathcal{P}[F](z)$ be a harmonic mapping between the unit disk
$\mathbf U$ and the Jordan domain $\Omega$, such that $F\in C^{1,\varpi}(\mathbf{T})$. Then partial derivatives of $f$ have continuous extension to the boundary of the unit disk.
\end{lemma}
\begin{proof}
Denote by $\Psi'(t)$ the $\partial_t F(e^{it})$.
If $F$ is Lipschitz continuous, then $\Phi=\Psi'\in L^\infty (\mathbf
T)$, and by famous Marcel Riesz theorem (see e.g.
\cite[Theorem~2.3]{gar}), for $1<p<\infty$ there is a constant $A_p$
such that  $$\|H(\Psi')\|_{L^p(\mathbf T)}\le A_p\|\Psi'\|_{L^p(\mathbf
T)}.$$ It follows that $\tilde \Phi=H(\Psi')\in L^1$. Since $r f_r$ is
the harmonic conjugate of $f_\tau$,  according to \eqref{lit}, we
have $rw_r = \mathcal{P}[H(\Psi')]$, and by using again the Fatou's theorem  we
have
\begin{equation}\label{amp}\lim_{r\to 1^-}f_r(re^{i\tau})
= H(\Psi')(\tau)\,\, (a.e.)\end{equation} By \eqref{hilbert}, by following the proof of Privaloff theorem \cite{ZY}, we obtain that if $|\Psi'(x)-\Psi'(y)|\le \omega(|x-y|)$ for the Dini continuous function, then $$|H(\Psi')(x+h)-H(\Psi')(x)|\le A\int_0^{2h} \frac{\omega(t)}{t} dt+B h \int_{h}^{2\pi} \frac{\omega(t)}{t^2} dt+C\omega(h),$$ for some absolute constants $A$, $B$ and $C$.  The detailed proof of the last fact can be found in Garnet book (see \cite[Theorem~III~1.3.]{gar}). This implies that $rw_r(re^{it})$ and $f_t(re^{it})$ have continuous extension to the boundary and this is what we needed to prove.
\end{proof}

We now prove the following lemma needed in the sequel
\begin{lemma}\label{eremenko}
Let $A$ be a positive integrable function in $[0,B]$ and assume that $q,Q>0$. Then there exists a continuous  increasing  function $\chi$ of $(0,+\infty)$ into itself depending on $A$, $B$, $q$ and and $Q$ such that $\lim_{x\to \infty}\chi(x)=\infty$, the function $g(x)=x\chi(x)$ is convex and  $$\int_0^BA(x) \chi(Qx^{-q})dx\le 4 \int_0^B A(x) dx$$ holds.
\end{lemma}
\begin{proof}
First define inductively a sequence $x_0=B$, $x_k>0$, $k> 0$ such that $x_{k+1}<x_{k}/2$, and
$$\int_0^{x_k}A(x)dx\le M2^{-k},$$ where $$M=\int_0^{B}A(x)dx.$$
This is possible because $A$ is integrable.

Then define a continuous function $\xi$ in $[0,B]$ by $\xi(x_k)=k$ and $\xi$ is linear on each interval $[x_{k+1},x_k]$, that is
$$\xi(x)=k+\frac{x_k-x}{x_{k} - x_{k+1}}, \ \ \ x\in [x_{k+1},x_k].$$
It is easy to see that this function is convex, decreasing and tends to $+\infty$ as $x\to \infty$. Moreover
$$\int_0^BA(x)\xi(x)dx\le M\sum_{k=0}^\infty (k+1)2^{-k}=4M.$$
Now set $\chi(x)=\xi((Q/x)^{\tau})$, $\tau=1/q$, and it remains to verity that $x\chi(x)$ is convex. This we do by differentiation:
$$(x\chi(x))'=\xi(Q^\tau x^{-\tau})-Q^\tau \tau x^{-\tau}\xi'(Q^\tau x^{-\tau}).$$
Both summands are increasing, therefore $x\chi(x)$ is convex.
\end{proof}
\section{The proof of Theorem~\ref{main}}
 By assumption of the theorem, the derivative of an arc-length parametrization $g'$ has a Dini continuous modulus of continuity $\omega$. Consider two cases. {\bf (i) $F(e^{it})=\Psi(t)\in C^{1,\varpi}(\mathbf{T})$.} Then by Lemma~\ref{jacabove2} the mapping $f(z)=\mathcal{P}[F](z)$ is $C^1$ up to the boundary.   Notice first that, if $L=\sup|\Psi'(t)|$, then it is clear that $L<\infty$. We will prove more. We will show that $L$ is bounded by a constant not depending a priory on $F$.   According to Lemma~\ref{jacabove2}, and to \eqref{distortion} we have
\begin{equation}\label{distortion1}\begin{split}|Df(e^{i\varphi})|^2&=(|f_z(e^{i\varphi})|+|f_{\bar z}(e^{i\varphi})|)^2\\&=\lim_{z\to e^{i\varphi}}(|f_z(z)|+|f_{\bar z}(z)|)^2\\&\le K \lim_{z\to e^{i\varphi}}(|f_z(z)|^2-|f_{\bar z}(z)|^2)\\ &=K(|f_z(e^{i\varphi})|^2-|f_{\bar z}(e^{i\varphi})|^2)=KJ_f(e^{i\varphi}).\end{split}\end{equation} Further \begin{equation}\label{dirihle}
|D f(re^{i\varphi})|=\sup_{|\xi|=1} |Dw(re^{i\varphi})\xi|\ge  |Dw(re^{i\varphi})(ie^{i\varphi})|=|\partial_\varphi f(re^{i\varphi})| .\end{equation} This implies that \begin{equation}\label{apoti}
|D f(e^{i\varphi})|^2 \ge {|\partial_\varphi
f(e^{i\varphi})|^2}= |\Psi'(\varphi)|^2.
\end{equation}
From (\ref{jakobo})
 (\ref{apoti}) and \eqref{distortion1},  we obtain: $$|\Psi'(\varphi)|^2\le
KC_1|\Psi'(\varphi)|\int_{-\pi}^\pi \frac 1{x^2}
{\int_0^{\rho(x,\varphi)}\omega(\tau)d\tau}
dx,$$ where $$\rho(x,\varphi)=d_\gamma(F(e^{i(\varphi+x)}),F(e^{i\varphi}))$$
i.e.
$$|\Psi'(\varphi)|\le K C_1 \int_{-\pi}^\pi \frac
{\rho(\varphi,x)}{x^2}
{\int_0^{1}\omega(\tau
\rho(\varphi,x))d\tau} dx.$$
Thus
$$|\Psi'(\varphi)|\le KC_1\int_{-\pi}^\pi \frac
{\rho(\varphi,x)}{x^2} \omega(
\rho(\varphi,x)) dx.$$
Let
\begin{equation}\label{maxsup}L:=\max_{x\in[0,2\pi]}|\Psi'(x)|=\max_{x\in[0,2\pi]}
\psi'(x)=\psi'(\varphi).\end{equation}
Then
$$L\le KC_1\int_{-\pi}^\pi \frac
{\rho(\varphi,x)}{x^2} \omega(
\rho(\varphi,x)) dx.$$
Further
$$M:=\frac{L}{2\pi KC_1}\le \int_{-\pi}^\pi M(x,\varphi)\frac{dx}{2\pi},$$ where $$M(x,\varphi)= \frac
{\rho(\varphi,x)}{x^2} \omega(
\rho(\varphi,x)).$$
The idea is to make use of the convex function constructed in Lemma~\ref{eremenko}, which depends on $K$ and $\Omega$ only to be found in the sequel.

Assume that $\chi:\mathbf{R^+}\to \mathbf{R^+}$ is a continuous increasing function to be determined in the sequel such that the function  $\Phi(t)=t\chi(t)$ is convex. By using Jensen inequality to the previous integral w.r.t. convex function $\Phi$
we obtain $$\Phi(M)\le \int_{-\pi}^\pi
\Phi(M(x,\varphi))\, \frac{dx}{2\pi},$$ i.e. \begin{equation}\label{mc}M\chi(M)\le  \int_{-\pi}^\pi
M(x,\varphi) \chi( M(x,\varphi))\, \frac{dx}{2\pi}.\end{equation}
In order to continue, we make use of \eqref{24march} from where and \eqref{maxsup} we infer that
\begin{equation}\label{ariim}
\rho(\varphi,x)\le B_\gamma L |x|.
\end{equation}
On the other hand by using Lemma~\ref{newle} we have

\begin{equation}\label{dritaime}
\rho(\varphi,x)\le B_\gamma \Lambda_\gamma(K) |x|^\alpha.
\end{equation}

This implies that \begin{equation}\label{afe}
M(x,\varphi)= \frac
{\rho(\varphi,x)}{x^2} \omega(
\rho(\varphi,x))\le \frac{ B_\gamma L}{x}\omega(
B_\gamma \Lambda_\gamma(K) |x|^\alpha).
\end{equation}
and
\begin{equation}\label{afeq}
M(x,\varphi)= \frac
{\rho(\varphi,x)}{x^2} \omega(
\rho(\varphi,x))\le \frac{ B_\gamma \Lambda_\gamma(K)}{x^{2-\alpha}}\omega(
B_\gamma \Lambda_\gamma(K) |x|^\alpha).
\end{equation}
So in view of Definition~\ref{defi} we have
\begin{equation}\label{afeqa}
M(x,\varphi)\le \frac{ B_\gamma \Lambda_\gamma(K)}{x^{2-\alpha}}\omega(|\gamma|).
\end{equation}
From \eqref{mc} and \eqref{afe}, we obtain
\begin{equation}\label{bela}\begin{split}
\chi\left(\frac{L}{2\pi KC_1}\right)&\le \int_{-\pi}^\pi \frac{  K C_1B_\gamma }{x}\omega(
B_\gamma \Lambda_\gamma(K) |x|^\alpha) \chi \left(\frac{ B_\gamma \Lambda_\gamma(K)\omega(|\gamma|)}{|x|^{2-\alpha}}\right)dx\\&=
 2\int_{0}^\pi \frac{  K C_1B_\gamma }{x}\omega(
B_\gamma \Lambda_\gamma(K) |x|^\alpha) \chi \left(\frac{ B_\gamma \Lambda_\gamma(K)\omega(|\gamma|)}{|x|^{2-\alpha}}\right)dx
\\&= \frac{2 K C_1B_\gamma }{B_\gamma \Lambda_\gamma(K) \alpha}\int_{0}^{B} \frac{\omega(y)}{y} \chi \left(Q y^{1-2/\alpha}\right)dy,
\end{split}
\end{equation}
where $$B=B_\gamma \Lambda_\gamma(K) \pi^\alpha$$ and $$Q={ \omega(|\gamma|)}{(B_\gamma \Lambda_\gamma(K))^{2-2/\alpha}}.$$ In view of the last term of \eqref{bela}, now it is the time to determine the function $\chi$.  Lemma~\ref{eremenko}, where $q=2/\alpha-1$, and $A(y)=\omega(y)/y$, provides us a function $\chi$ such that $\Phi$ is convex and there holds the estimate $$\int_{0}^{B} \frac{\omega(y)}{y} \chi \left(Q y^{1-2/\alpha}\right)dy\le
  4 \int_{0}^{B} \frac{\omega(y)}{y} dy.$$
From \eqref{bela} we have $$\chi\left(\frac{L}{2\pi KC_1}\right)\le \frac{8 K C_1B_\gamma }{B_\gamma \Lambda_\gamma(K) \alpha}\int_{0}^{B} \frac{\omega(y)}{y}dy=:\Upsilon(K,\Omega).$$ Since $\chi$ is increasing we infer finally that

\begin{equation}
{L}\le {2\pi KC_1}\cdot \chi^{-1}(\Upsilon(K,\Omega))=\frac{\pi^2}{2}{ K}\cdot \chi^{-1}(\Upsilon(K,\Omega)).
\end{equation}
By the maximum principle, for $z=re^{i\varphi}$, we further have $$|\partial_\varphi f(z)|\le L.$$ Since $f$ is $K-$quasiconformal, we have $$|Dw(z)|\le K |\partial_\varphi f(z)|.$$ This and Mean value inequality implies that \begin{equation}\label{mvi}
|f(z)-f(z')|\le K L |z-z'|, \ \ \ |z|<1,|z'|<1.
\end{equation}
{\bf (ii)} $F\notin C^{1,\varpi}(\mathbf{T})$.
In order to deal with non-smooth $F$,  we make use of approximate argument.
We begin by this definition.
\begin{definition}
{\it Let $G$ be a domain in $\Bbb C$ and let $a\in\partial G$. We
will say that $G_a\subset G$ is a neighborhood of $a$ if there
exists a disk $D(a,r):=\{z:|z-a|<r\}$ such that $D(a,r)\cap
G\subset G_a$.}
\end{definition}
 Let $t=e^{i x}\in \mathbf{T}$, then $F(t)=\Psi(x)\in
 \partial \Omega$. Let $g$ be an arc-length parametrization of $\partial\Omega$
with $g(\psi(x))=F(e^{ix})$, where $\psi:[0,2\pi]\to [0,|\gamma|]$ as in the first part of the proof. Put $s=\psi(x)$.  Since modulus of continuity of $g'$ is a Dini continuous function $\omega$, there
exists a neighborhood $\Omega_t$ of $\Psi(t)$, such that the derivative of its arc-length parametrization $g'_t$ has modulus of continuity $C_t\cdot \omega$.  Moreover there
exist positive numbers $r_t$ and $R_t$ such that,
\begin{eqnarray}\label{subdomains}
  \Omega^\tau _t:=\Omega_t +
ig'(s)\cdot \tau &\subset& \Omega, \ \ \ \tau\in (0,R_t) \\
   \partial
\Omega^\tau_t &\subset& \Omega, \ \ \ \tau\in (0,R_t) \\
  g[s-r_t,s+r_t] &\subset& \partial\Omega_t.
\end{eqnarray}
 An example of a family $\Omega^\tau_t$
such that $\partial\Omega^\tau_t \in C^{1,\alpha}$, $0<\alpha<1$ and with the property
\eqref{subdomains} has been given in \cite{kalajd}. The same construction yields the family $\partial\Omega^\tau_t$ with the above mentioned properties.

Take $U_\tau = f^{-1}(\Omega_t^\tau)$. Let
$\eta_t^\tau $ be a conformal mapping of the unit disk onto $U_\tau$ with normalized boundary condition: $\eta_t^\tau(e^{i2k\pi/3})=f^{-1}(\zeta_k)$, $k=0,1,2$, where $\zeta_0,\zeta_1,\zeta_2$ are three points of $\partial\Omega_t^\tau$ of equal distance. Then the mapping
$$f_t^\tau(z) := f(\eta_t^\tau(z))-ig'(s)\cdot \tau$$ is a
harmonic $K$ quasiconformal mapping of the unit disk onto $\Omega_t$
satisfying the boundary normalization. Moreover
$$f_t^\tau=\mathcal{P}[F_t^\tau]\in C^1(\overline {\mathbf U}),$$ for some function $F_t^\tau\in C^{1}(\mathbf{T})$.

Since $[0,l]$ is compact, there exists a finite family of Jordan
arcs $$\gamma_j=g(s_j-r_{s_j}/2,s_j+r_{s_j}/2),\, j=1,\dots,n,$$
covering $\gamma$ and assume that $F(t_j)=s_j$. Let $$F_{j,\tau}:=F_{t_j}^\tau,\ \  a_{j,\tau}:=\eta_{t_j}^\tau \text{ and }f_{j,\tau}:=f_{t_j}^\tau.$$

Using the case "$F\in C^{1,\varpi}$ " it follows that there exists
a constant $C'_{j}=C'(K,\gamma_{j})$ such that
$$|\partial_\varphi F_{j,\tau}'(e^{i\varphi})|\le C'_{j}$$ and

\begin{equation}\label{wjh}|f_{j,\tau}(z_1)-f_{j,\tau}(z_2)|\le
KC'_j|z_1-z_2|.\end{equation}

Since $a_{j,\tau}(z)$ converges uniformly on compact subsets of
$\mathbf U$ to the function $a_{j,0}(z)$ when $\tau\to 0$, and since
$f_{j,\tau}=f\circ a_{j,\tau}$, (\ref{wjh}) implies
\begin{equation}\label{wh}|f_{j}(z_1)-f_{j}(z_2)|\le
KC'_j|z_1-z_2| \text{ for $z_1,z_2\in \overline{\mathbf U}$},
\end{equation} where $f_j=f\circ a_{j,0}=\mathcal{P}[F_j]$. For $z_1=e^{it}$
and $z_2=e^{i\varphi}$, $t\to \varphi$ we obtain that
$|\partial_\varphi F_j(e^{i\varphi})|\le KC'_j$ $a.e.$ Since the mapping
$b_j=a_{0,j}^{-1}$ can be extended conformally across the arc
$S_j=f^{-1}(\lambda_j)$, where
$\lambda_j=g(s_j-t_{s_j},s_j+t_{s_j}),\, $ there exists a constant
$L_j$ such that $|b_j(z)|\le L_j$ on $S_j'=\mathbf{T}\cap
f^{-1}(\gamma_j),$ $j=1,\dots,n$. Hence $|\partial_\varphi F(e^{i\varphi})|\le
KC'_j\cdot L_j$ on $S_j'$. Let $C'=\max\{KC'_j\cdot
L_j:j=1,\dots,n\}$. (\ref{desired}) and (\ref{equati}) easily
follow from $\mathbf{T}=\bigcup_{j=1}^{n} S_j'$.

The proof is completed.

\section*{Acknowledgement}
I am grateful to professor Alexandre Eremenko, whose idea is used in the proof of Lemma~\ref{eremenko}, and for his numerous useful suggestions and corrections that have improved substantially this paper.

\end{document}